\documentclass[12pt, twoside]{article}
\usepackage{amsmath,amsthm,amssymb}
\usepackage{times}
\usepackage{enumerate}
\usepackage{url}
\usepackage{lineno}
\usepackage{microtype}
\usepackage[american]{babel}
\usepackage{indentfirst}
\usepackage{color}
\usepackage{geometry}
\usepackage{epsfig}
\def\titlerunning#1{\gdef\titrun{#1}}
\makeatletter
\def\author#1{\gdef\autrun{\def\and{\unskip, }#1}\gdef\@author{#1}}
\def\address#1{{\def\and{\\\hspace*{18pt}}\renewcommand{\thefootnote}{}%
\footnote {#1}}%
\markboth{\autrun}{\titrun}}
\makeatother
\def\email#1{e-mail: #1}
\def\subjclass#1{{\renewcommand{\thefootnote}{}%
\footnote{\emph{Mathematics Subject Classification (2010):} #1}}}
\def\keywords#1{\par\medskip
\noindent\textbf{Keywords.} #1}

\newtheorem{result}{\textbf{Theorem}}

\newtheorem{proposition}{Proposition}[section]
\newtheorem{lemma}{Lemma}[section]
\newtheorem{definition}{Definition}[section]

\newtheorem{remark}{Remark}[section]

\numberwithin{equation}{section}

\newcommand{\R}{\mathbb R}

\begin{document}
\baselineskip=15pt

\titlerunning{Nonlinear and  degenerate discounted approximation in discrete weak KAM theory}
\title{Nonlinear and  degenerate discounted approximation in discrete weak KAM theory}
\author{Panrui Ni\and Maxime Zavidovique}

\maketitle
\address{Panrui Ni: Sorbonne Universit\'e, Universit\'e de Paris Cit\'e, CNRS, Institut de Math\'ematiques de Jussieu-Paris Rive Gauche, Paris 75005, France;
\email{panruini@imj-prg.fr}
\and Maxime Zavidovique: Sorbonne Universit\'e, Universit\'e de Paris Cit\'e, CNRS, Institut de Math\'ematiques de Jussieu-Paris Rive Gauche, Paris 75005, France;
\email{mzavidovi@imj-prg.fr}}


\vspace{-6ex}

\begin{flushright}
\emph{\`A la mémoire de Nicolas Bergeron\\
un collègue en or\\
un mathématicien sensationnel\\
une incarnation de générosité.}
\end{flushright}

\begin{abstract}

In this paper, we introduce a discrete version of the nonlinear implicit Lax-Oleinik operator as studied for instance in \cite{WWY}. We consider the associated vanishing discount problem with a non-degenerate condition and prove convergence of solutions as the discount factor goes to $0$. We also discuss the uniqueness of the discounted solution. This paper can be thought as the discrete version of \cite{CFZZ}, and a generalization of \cite{DFIZ} and \cite[Chapter 3]{Z}. The convergence result is a selection principle for fixed points of a family of nonlinear operators.

\end{abstract}





\section*{Introduction}

\subsection{Brief history of the problem}
The discounted approximation appeared for Hamilton-Jacobi equations in Lions Papanicolaou and Varadhan's celebrated preprint \cite{LPV}. The goal is to solve\footnote{In this introduction, solutions are intended in the viscosity sense but no knowledge of this notion is required to understand the paper.} an equation on the torus of the form
\begin{equation}\label{HJ0}
H(x,d_x u) = c_0, \quad x\in \mathbb T^N
\end{equation}
where  $u : \mathbb T^N \to \R$ and $c_0\in \R$ are the unknown and the Hamiltonian $H : \mathbb  T^N \times \R^N\to \R$ satisfies some coercivity condition. The problem of the previous equation is that if it admits solutions, they are not unique (the set of solutions is invariant by addition of constants). To deal with this problem, the authors approximate the equation as follows: given $\lambda >0$, they solve the discounted equation
$$\lambda u_\lambda (x) + H(x,d_x u_\lambda) = 0,$$
and establish that the solution $u_\lambda $ is unique, that the family $(u_\lambda)_{\lambda \in (0,1)}$ is equicontinuous, that $\lambda u_\lambda$ converges to a constant ($-c_0$) as $\lambda \to 0^+$ and that $(u_\lambda + c_0/\lambda)_{\lambda \in (0,1)}$ is bounded. Therefore, taking a converging subsequence $u_{\lambda_n}+ c_0/\lambda_n\to u_0$ provides a solution $u_0$ to \eqref{HJ0}.

The first convergence result for the whole family $(u_\lambda + c_0/\lambda)_{\lambda \in (0,1)}$ was established in \cite{IS} for a particular case followed by \cite{DFIZ2} for a result in full generality under Tonelli type hypotheses on $H$. A discrete version of this result was published in \cite{DFIZ} around the same time. The necessity of convexity of Hamiltonian in the convergence result is given by a counterexample in \cite{Zi}.

This convergence phenomenon was followed by many generalizations (see \cite{DD,IS20} for non-compact cases, see \cite{IMT,IMT2,MT,Zhang} for second order cases, see \cite{DZ,I,IJ} for weakly coupled systems, and see \cite{CP} for mean field games). Amongst the ones that are of interest to us here, let us cite also the papers the papers \cite{Chen,CCIZ,WYZ, Su} that prove similar results for equations of the form $G\big(x, \lambda u_\lambda(x), d_x u_\lambda (x)\big)=0$ where $G(x,u,p)$ verifies Tonelli type hypotheses in the variables $(x,p)$ and is increasing in $u$. This is the nonlinear version of the problem, the results in  \cite{DFIZ2} corresponding to the particular case $G(x,u,p) = u+ H(x,p)$.
The degenerate aspect was studied in \cite{ZAP} for Hamiltonians of the form $G(x,u,p) = \alpha(x)u+ H(x,p)$ where $\alpha$ is a continuous nonnegative function that verifies some non degeneracy condition but is allowed to vanish on large portions of $\mathbb T^N$ (a discrete version of this results is presented in \cite{Z}). Both those settings were merged in a nonlinear degenerate setting in \cite{CFZZ} where general Hamiltonians  $G(x,u,p)$ are considered, verifying Tonelli type hypotheses in the variables $(x,p)$ and being non-decreasing in $u$. The nondegeneracy hypothesis consists in prescribing that $G$ is increasing in some regions.  When the equation is not non-decreasing in $u$, the asymptotic behavior is not clear yet. One can refer to \cite{DW,WYZ2} for the approximation process when $\lambda\to 0^-$, and see \cite{N} for the asymptotic behavior of a particular non-monotone case.

\subsection{The discretization }
The philosophy of the discrete problem stems from  Lax-Oleinik type formulas. In the previously mentioned results, if $u_\lambda $ solves $G\big(x, \lambda u_\lambda(x), d_x u_\lambda (x)\big)=0$ then the solution $u_\lambda$ verifies
$$\forall x\in \mathbb T^N, \ \ \forall t>0, \quad u_\lambda(x) = \inf_\gamma u_\lambda \big(\gamma(-t)\big) + \int_{-t}^0 L_G\big(\gamma(s), u_\lambda \big(\gamma(s)\big),\dot \gamma(s)\big) ds,$$
where $L_G : \mathbb T^N \times \R \times \mathbb R^N$ is a function related to $G$ that is non-increasing in $u$, and the infimum is taken amongst absolutely continuous curves $\gamma : \mathbb T^N \to \R$ such that $\gamma(0)=x$. The idea is then to fix $t>0$ (small in spirit) and to consider an approximation of the integral by a function that may depend on $\gamma(0)=x$ and $\gamma(-t)$ but also possibly on the values $u_\lambda(x)$ and $u_\lambda\big(\gamma(-t)\big)$.  When the approximation function depends linearly on $u_\lambda\big(\gamma(-t)\big)$, and does not depend on $u_\lambda(x)$, the whole discrete system then reduces to the case considered in \cite{DFIZ} and \cite[Chapter 3]{Z}. One can refer to the example given in Section \ref{secdisc} below.

\subsection{Setting and statement of results}

One advantage of the discretization is that it is non longer necessary to have a differentiable structure. We then work on $(X,d)$ a compact metric space and consider a continuous function that is $C^1$ with respect to the last two variables\footnote{Our results actually require less regularity.}, $\ell : X\times X \times \R \times \R$ such that

\begin{itemize}
\item [(1)] there is a constant $\kappa_u>0$ such that for all $(x,y,u,v)$,  $0\geqslant \partial_u \ell (x,y,u,v) \geqslant -\kappa_u$,
\item [(2)] there is a constant $\kappa_v>0$ such that for all $(x,y,u,v)$,  $0\geqslant \partial_u \ell (x,y,u,v) \geqslant -\kappa_v$,
\item [(3)] $\int_{X\times X}\big(\partial_u \ell(z,x,0,0)+\partial_v \ell(z,x,0,0)\big)d\mu(z,x)<0$ for all Mather measures $\mu$ of $\ell(\cdot,\cdot,0,0)$.
\end{itemize}

The notion of Mather measure will be detailed later in the paper. Let us already stress that it may happen (quite often actually) that there is only one Mather measure. Therefore if this is the case, this last nondegeneracy condition only requires that $\partial_u \ell(z,x,0,0)$ or $\partial_v \ell(z,x,0,0)$ is negative somewhere on the support  of this Mather measure.

Let us denote $c_0\in \R$ the critical constant of the function $\ell(\cdot,\cdot,0,0)$ (its precise definition is given later). Given this cost function (or discrete Lagrangian) we introduce an implicit Lax-Oleinik operator:

\begin{proposition}
There is a $\lambda_0>0$ such that for $0<\lambda<\lambda_0$, if $\varphi \in C^0(X,\R)$ there is a unique
$T_\lambda \varphi \in C^0(X,\R)$ such that for all $x\in X$,
\[T_\lambda \varphi(x)=\min_{z\in X}\big\{\varphi(z)+\ell\big(z,x,\lambda \varphi(z),\lambda T_\lambda \varphi(x)\big)\big\}+c_0.\]
\end{proposition}

The implicit Lax-Oleinik semigroup was studied in \cite{WWY} in the continuous setting. It corresponds to the viscosity solutions of Hamilton-Jacobi equations depending on the unknown function. It is also meaningful in the optimal control theory of systems with a non-holonomic constraint, see \cite{CCJWY}. Thus, our discrete semigroup here can be thought of as an approximation of the cost function of a class of optimal control systems.

We then solve the discounted equation:

\begin{result}
For $\lambda \in (0,\lambda_0)$ the operator $T_\lambda$ has a fixed point $u_\lambda$. Moreover if we set $S_\lambda$ the set of fixed points of $T_\lambda$, then the family $(S_\lambda )_{\lambda \in (0,\lambda_0)} $ is made of equicontinuous and equibounded functions.
\end{result}

Finally we prove the convergence of solutions of the discounted equations:

\begin{result}
The family $(S_\lambda )_{\lambda \in (0,\lambda_0)} $ converges to a singleton as $\lambda \to 0$ in the sense that there exists $u_0 : X\to \R$ such that for any choice $u_\lambda \in S_\lambda$ for $\lambda \in (0,\lambda_0)$ the (uniform) convergence $u_\lambda \to u_0$ holds as $\lambda \to 0^+$.
\end{result}

In establishing those results we also prove two characterizations  for the limit $u_0$. We also address the issue of uniqueness of fixed points $u_\lambda \in S_\lambda$ under quite natural assumptions. Actually, in the simplified setting presented above, we prove that $S_\lambda$ is a singleton for $\lambda$ small enough.

\subsection{Organization of the paper}

\begin{itemize}
\item In the first Section \ref{secwk} we recall some needed facts on discrete weak KAM solutions (corresponding to $\lambda = 0$).

\item In the following Section \ref{Sec2} we introduce a general theory of implicit Lax-Oleinik operators.

\item Then in Section \ref{secdisc} we define and study solutions to the discounted equations.

\item Finally in Section \ref{secconv} we prove the convergence as $\lambda \to 0$ of solutions to the discounted equations.

\item The last section adresses the uniqueness issue.

\end{itemize}

\section{Classical discrete weak KAM theory}\label{secwk}
We briefly recall classical results that will be used in the rest of the paper. References are, amongst many others, \cite{Z,Z1,DFIZ}. Let $(X,d)$ be a compact metric space and $\ell_0 : X\times X \to \R$ a continuous function sometimes called cost function. The discrete Lax-Oleinik semigroup is
\begin{definition}\rm
The discrete Lax-Oleinik semigroup is the operator $T_0 : C^0(X,\R) \to C^0(X, \R)$ which to $f : X\to \R$ associates
$$T_0 f : x \mapsto T_0f(x) = \min_{y\in X} f(y)+\ell_0(y,x).$$
\end{definition}
It can be checked that $T_0$ is non decreasing, $1$-Lipschitz for the sup-norm and commutes with addition of constants. This allows to prove the discrete weak KAM theorem:

\begin{result}\label{weakKAM}
There exists a unique constant $c_0$ such that there is a function $u : X\to \R$ verifying $u = T_0u + c_0$.
\end{result}
The constant $c_0$ is called the critical constant of $\ell_0$. A function $u$ verifying $u = T_0u + c_0$ is called a weak KAM solution. Weak KAM solutions are not unique, for instance, if $K\in \R$ then $u+K$ is also a weak KAM solution. Note also that the critical constant for the cost function $\tilde\ell_0 = \ell_0+c_0$ is $\tilde c_0=0$.

By definition of $T_0$, a weak KAM solution $u$ verifies $u(y)-u(x)\leqslant \ell_0(x,y) + c_0$ for all $x,y \in X$. This motivates the definition

\begin{definition}\rm
A function $v : X\to \R$ is called a subsolution if it verifies
$$\forall (x,y)\in X\times X , \quad v(y)-v(x) \leqslant  \ell_0(x,y) + c_0$$
This is equivalent to $T_0v+c_0 \geqslant v$.
\end{definition}
An easy but important fact is that
\begin{proposition}\label{conv}
The set of subsolutions is closed (under uniform convergence but also under pointwise convergence) and it is convex.
\end{proposition}

An important tool we will use is that of {\bf Mather measure}. In all the paper, all measures are Borel measures even if not explicitly stated.

\begin{definition}\rm A Borel measure $\mu$ on $X\times X$ is closed if its marginals coincide: $\pi_{1*} \mu = \pi_{2*}\mu$, where $\pi_1(x,y) = x$ and $\pi_2(x,y)=y$.

If $v : X\to \R$ is a continuous subsolution, integrating the family of inequalities  $v(y)-v(x) \leqslant  \ell_0(x,y) + c_0$ against a closed probability measure $\mu$, we discover that $\int_{X\times X} \ell_0(x,y)d\mu (x,y) \geqslant -c_0$. This leads to the notion of Mather measures (or equivalently of minimizing measures):

\begin{definition} \rm
A Mather measure (or minimizing measure) $\mu$ is a probability measure on $X\times X$ that is closed and verifies
$$\int_{X\times X} \ell_0(x,y) d\mu(x,y) = -c_0.$$
We will denote by $\mathfrak M_0$ the set of Mather measures.
\end{definition}

Finally we will need an important function associated to $\ell_0$ called Peierls' barrier. If $n>0$, let
$$\forall(x,y)\in X\times X , \quad h_n(x,y) = \min_{\substack{ (x_0,\cdots , x_n)\in X^{n+1} \\x_0=x, x_n=y}} \sum_{k=0}^{n-1} \ell_0(x_k,x_{k+1}).$$
\begin{definition}
Peierls' barrier is the function $h : X\times X \to \R$ defined by
$$\forall (x,y)\in X\times X, \quad h(x,y) = \liminf_{n\to +\infty} h_n(x,y)+nc_0.$$
\end{definition}

Here are some key properties of Peierls' barrier

\begin{proposition}\label{proph}
\begin{enumerate}
\item The function $h$ is finite valued and continuous on $X\times X$.
\item For all $x \in X$, the function $h(x,\cdot)$ is a weak KAM solution and the function $-h(\cdot , x)$ is a subsolution.
\item If $v : X\to \R$ is any subsolution, then
$$\forall (x,y)\in X\times X, \quad v(y)-v(x) \leqslant h(x,y).$$
\end{enumerate}
\end{proposition}

A crucial set in weak KAM theory is the projected Aubry set:

\begin{definition}\rm
The projected Aubry set is $\mathfrak A = \{x \in X , \ \ h(x,x)=0\}$.
\end{definition}
This set is proven to be non-empty.

We end this section with a comparison principle:

\begin{proposition}\label{comparison}
Let $u:X\to \R$ be a weak KAM solution and $v : X\to\R$ be a subsolution.
Assume that $u_{|\mathfrak A} \geq v_{|\mathfrak A}$ then $u\geqslant v$ on $X$.
\end{proposition}

\end{definition}

\section{Discrete version of implicit semigroup}\label{Sec2}

Assume $(X,d)$ is a compact metric space where $d : X\times X \to \R$ is the distance function.  In the paper, $c:X\times X\times \mathbb R^2\to \mathbb R$ is a continuous function. Hypotheses that will be needed are the following
\begin{itemize}
\item[(Lu)] for each $(z,x,v)\in X\times X\times \mathbb R$, $u\mapsto c(z,x,u,v)$ is $\kappa_u$-Lipschitz continuous and $\kappa_u\leq 1$.
\item[(Lv)] for each $(z,x,u)\in X\times X\times \mathbb R$, $v\mapsto c(z,x,u,v)$ is $\kappa_v$-Lipschitz continuous and $\kappa_v<1$.
\end{itemize}

\begin{proposition}\label{exT}
Assume $c$ verifies hypothesis {\rm (Lv)}. For each continuous function $\varphi:X\to\mathbb R$, there is a unique continuous function $T\varphi: X\to \mathbb R$ satisfying
\[T\varphi(x)=\min_{z\in X}c\big(z,x,\varphi(z),T\varphi(x)\big).\]

The operator $T : (C^0(X,\R),\|\cdot \|_\infty)\to (C^0(X,\R),\|\cdot \|_\infty)$  is continuous and compact.

Moreover, if the family of functions $x\mapsto c(z,x,u,v)$ is locally equi-Lipschitz continuous, $T\varphi(x)$ is Lipschitz continuous.
\end{proposition}
\begin{proof}
We first prove that $T\varphi(x)$ exists. For a continuous function $f: X\to \mathbb R$, define
\[\mathcal{A} f(x)=\min_{z\in X}c\big(z,x,\varphi(z),f(x)\big).\]
By the continuity of $c$ and $f$ and compactness of $X$, we see that $\mathcal A$ is an operator from $C^0(X,\R)$ to itself. Indeed, $\mathcal{A} f$ is an infimum of equicontinuous functions. We are going to find a fixed point of $\mathcal A$. We take two continuous functions $f$ and $g$ on $X$. By compactness of $X$, let $z$ be a minimal point realizing the minimum in the definition of $\mathcal Ag(x)$, then we have
\[\mathcal Af(x)-\mathcal Ag(x)\leq c\big(z,x,\varphi(z),f(x)\big)-c\big(z,x,\varphi(z),g(x)\big)\leq \kappa_v \|f-g\|_\infty.\]
Exchanging the role of $f$ and $g$, we get that $\mathcal A$ is a contraction in $(C^0(X,\R),\|\cdot\|_\infty)$, since $\kappa_v<1$. By the Banach fixed point theorem, there is a unique fixed point of $\mathcal A$, which is $T\varphi(x)$.

Then we prove the operator $T$ is compact. Let $r>0$ and $\varphi \in C^0(X,\R)$ such that $\|\varphi\|_\infty<r$. Consider the sequence $(f_n)_{n\in\mathbb N}$ with $f_0=0$ and
\[f_{n+1}(x)=\min_{z\in X}c\big(z,x,\varphi(z),f_n(x)\big).\]
Since $\mathcal A$ is a contraction, the sequence $f_n$ converges to $ T\varphi$ uniformly. We also have
\begin{align*}
\|f_{n+1}\|_\infty&\leq \sum_{i=0}^n\|f_{i+1}-f_i\|_\infty\leq \sum_{i=0}^n(\kappa_v)^i\|f_1\|_\infty
\\ &\leq \frac{1}{1-\kappa_v}\|f_1\|_\infty\leq \frac{1}{1-\kappa_v}\max_{\substack{z\in X, x\in X\\ |u|\leq r}}|c(z,x,u,0)|=:R_0.
\end{align*}
Define $R:=\max\{R_0,r\}$. Let $\omega $ be a modulus of continuity of $c$ restricted to the compact set $X\times X \times [-R,R]^2$. For each $x,y\in X$, let $z_y$ be a minimal point in the definition of $f_1(y)$, we have
\[f_1(x)-f_1(y)\leq c(z_y,x,\varphi(z_y),0)-c(z_y,y,\varphi(z_y),0)\leq \omega\big(d(x,y)\big).\]
Exchanging $x$ and $y$, we have
\[|f_1(x)-f_1(y)|\leq \omega\big(d(x,y)\big).\]

More generally, if $f\in C^0(X,\R)$ is such that $\|f\|_\infty <R$, and if $\omega_f$ is a modulus of continuity of $f$, then for $x,y\in X$,

\begin{multline*}
\mathcal Af(x) - \mathcal Af(y) \leqslant c\big(z_y,x,\varphi(z'_y),f(x)\big)-c\big(z'_y,y,\varphi(z_y),f(y)\big) \\
=   c\big(z_y,x,\varphi(z'_y),f(x)\big) -c\big(z_y,x,\varphi(z'_y),f(y)\big)\\
 +c\big(z_y,x,\varphi(z'_y),f(y)\big) -c\big(z'_y,y,\varphi(z_y),f(y)\big) \\
\leqslant\kappa_v \omega_f(x,y) + \omega(x,y).
\end{multline*}
It follows, by exchanging the roles of $x$ and $y$, that $\kappa_v \omega_f + \omega$ is a modulus of continuity of $\mathcal A f$. Applying to the sequence $(f_n)_n$ we obtain by induction that $f_{n+1}$ has $(1+\kappa_v+\cdots +\kappa_v^n)\omega$ as modulus of continuity. Hence the whole sequence is equicontinuous with modulus $\frac{\omega}{1-\kappa_v}$ and so is $T \varphi$. As $\|T\varphi\|_\infty \leqslant R_0$ and $R_0$ only depends on $r$, this proves that $T$ is compact by the Arzela-Ascoli theorem. Next we prove that $T$ is continuous. Let $(\varphi_n)_n$ be a sequence converging to $\varphi$. By the previous point, the sequence $(T \varphi_n)_n$ is precompact. Let $(k_n)_n$ be an extraction such that $(T \varphi_{k_n})_n$ uniformly converges to a function $\psi$. Then if $x\in X$ we can pass to the limit in the relations
\[T\varphi_{k_n}(x)=\min_{z\in X}c\big(z,x,\varphi_{k_n}(z),T\varphi_{k_n}(x)\big)\] to obtain
\[\psi(x)=\min_{z\in X}c\big(z,x,\varphi(z),\psi(x)\big)\] and by uniqueness, $\psi = T\varphi$. This proves that $T$ is continuous.

Now we prove the Lipschitz continuity of $T\varphi$ under the additional Lipschitz assumption of $c$ with respect to $x$. Let $\kappa^R_x$ be the Lipschitz constant of $x\mapsto c(z,x,u,v)$ for $|u|$ and $|v|$ bounded by $R>0$. Applying the previous method, we obtain that if $f_n$ is $\kappa_n $-Lipschitz, then $f_{n+1} $ is $\kappa_{n+1}$-Lipschitz with $\kappa_{n+1}=\kappa^R_x+\kappa_v\kappa_n$. Therefore, $(f_n)_n$ is equi-Lipschitz continuous with constant $\frac{\kappa^R_x}{1-\kappa_v}$, and uniformly converges to $T\varphi$. Then $T\varphi$ is Lipschitz continuous. \end{proof}

Now we assume
\begin{itemize}
\item[(M)] $u\mapsto c(z,x,u,v)$ is non-decreasing and $v\mapsto c(z,x,u,v)$ is non-increasing.
\end{itemize}
\begin{proposition}(Order preserving). \label{ord}
If $f\leq g$, then $Tf\leq Tg$.
\end{proposition}
\begin{proof}
We argue by contradiction. Assume there is $x\in X$ such that $Tf(x)>Tg(x)$. Let $z_g$ be a  point realizing the definition of $Tg(x)$, then we have
\[Tf(x)\leq c\big(z_g,x,f(z_g),Tf(x)\big)\leq c\big(z_g,x,g(z_g),Tg(x)\big)=Tg(x),\]
which leads to a contradiction.
\end{proof}

 To end this section, assume now that the three hypotheses (Lu), (Lv) and (M) are satisfied:

\begin{proposition}\label{nonex}(Non-expensiveness).
For each $f,g$, we have $\|Tf-Tg\|_\infty\leq \|f-g\|_\infty$.
\end{proposition}
\begin{proof}
We are going to prove $Tf(x)-\|f-g\|_\infty-Tg(x)\leq 0$ for each $x\in X$. We argue by contradiction. Assume there is $x\in X$ such that $Tf(x)-\|f-g\|_\infty-Tg(x)>0$. Let $z_g$ be a  point realizing the definition of $Tg(x)$. Then we have
\begin{align*}
&Tf(x)-\|f-g\|_\infty-Tg(x)
\\ & \leq c\big(z_g,x,f(z_g),Tf(x)\big)-\|f-g\|_\infty-c\big(z_g,x,g(z_g),Tg(x)\big)
\\ & \leq c\big(z_g,x,g(z_g),Tf(x)\big)-c\big(z_g,x,g(z_g),Tg(x)\big)\leq 0,
\end{align*}
which leads to a contradiction. For the last inequality, we use (Lu). Exchanging $f$ and $g$, and then the proof is complete.
\end{proof}


\section{Discounted solutions}\label{secdisc}

Now we consider the discounted problem. Let $\ell:X\times X\times \mathbb R^2\to\mathbb R$ be continuous and satisfy
\begin{itemize}
\item [(l1)] $\ell(z,x,u,v)$  is  $\kappa_u$-Lipschitz in $u$ and $\kappa_v$-Lipschitz in $v$.
\item [(l2)] $\ell(z,x,u,v)$ is non-increasing in $u$ and $v$.
\item [(l3)] $\partial_u \ell(z,x,0,0)$ and $\partial_v \ell(z,x,0,0)$ exist, and \[|\ell(z,x,v,u)-\ell(z,x,0,0)-\partial_u \ell(z,x,0,0)u-\partial_v \ell(z,x,0,0)v|\leq \frac{|u|+|v|}{2}\eta(\frac{|u|+|v|}{2}),\] where $\eta$ is a modulus of continuity.
\item [(l4)] $\int_{X\times X}\big(\partial_u \ell(z,x,0,0)+\partial_v \ell(z,x,0,0)\big)d\mu(z,x)<0$ for all Mather measures $\mu$ of $\ell(\cdot,\cdot,0,0)$.
\end{itemize}
{\bf We assume the critical value of $(x,z)\mapsto \ell(z,x,0,0)$ equals zero\footnote{If this is not the case,  all our results apply to the function   $\tilde \ell = \ell-c_0$ where $c_0$ is the critical constant of $(x,z)\mapsto \ell(z,x,0,0)$.}.}

Note that under hypothesis (l4), both functions $(z,x) \mapsto \partial_u \ell(z,x,0,0)$ and $(z,x) \mapsto \partial_v \ell(z,x,0,0)$ are continuous as uniform limits of continuous functions.

Let $\lambda_0>0$ such that $\lambda_0 \max(\kappa_u,\kappa_v) < 1$.  For $\lambda_0> \lambda>0$, $u+\ell(z,x,\lambda u,\lambda v)$ satisfies the basic assumptions (Lu), (Lv) and (M) for $c(z,x,u,v)$ in Section \ref{Sec2}.
 Define for $\varphi \in C^0(X,\R)$
\[T_\lambda \varphi(x)=\min_{z\in X} \big\{\varphi(z)+\ell\big(z,x,\lambda \varphi(z),\lambda T_\lambda \varphi(x)\big)\big\}.\]
This is well defined by Proposition \ref{exT}.

\underline{\bf Example}: If $\ell(z,x,u,v)$ is of the form $\ell_0(z,x)-\alpha(z)u$, where $\alpha : X\to \R$ is a non-negative function, we have
\[T_\lambda\varphi(x)=\min_{z\in X}\big\{\big(1-\lambda \alpha(z)\big)\varphi(z)+\ell_0(z,x)\big\}=T_0\big((1-\lambda\alpha)\varphi\big)(x),\]
which is the degenerate vanishing discount problem as treated in \cite{Z}.

\begin{proposition}\label{bdd}
For $\lambda<\lambda_0$, the operator $T_\lambda$ admits at least one fixed point $u_\lambda$.

Moreover,  the  family of all such fixed points $(u_\lambda)_{\lambda\in (0,\lambda_0)}$ is uniformly bounded and equi-continuous.
\end{proposition}
\begin{proof}
{\bf Step 1.} We first prove the existence of $u_\lambda$ for $\lambda< \lambda_0$. By the discrete weak KAM theorem, $T_0$ has a fixed point $u$ (recall the critical constant is $0$). Since
\[T_0(u+k)=T_0u+k=u+k,\quad \forall k\in\mathbb R,\]
we can choose $\bar u\geq 0$ with $T_0\bar u=\bar u$. We prove that $T_\lambda \bar u\leq \bar u$. Assume there is a point $x\in X$ such that $T_\lambda \bar u(x)>\bar u(x)$. Let $z$ be a point realizing the minimum in the definition of $T_0\bar u(x)$, we have
\[T_\lambda\bar u(x)\leq \bar u(z)+\ell\big(z,x,\lambda \bar u(z),\lambda T_\lambda \bar u(x)\big)\leq \bar u(z)+\ell(z,x,0,0)=T_0\bar u(x)=\bar u(x),\]
which leads to a contradiction. 
By Proposition \ref{ord}, we have
\[\bar u\geq T_\lambda \bar u\geq T_\lambda\circ T_\lambda \bar u\geq \dots.\]
Similarly, let $\underline u$ be a negative weak KAM solution, we have
\[\underline u\leq T_\lambda \underline{u}\leq T_\lambda\circ T_\lambda \underline{u}\leq \dots.\]
Since $\underline{u}\leq \bar u$, we have for all $n\geqslant 0$,
\[\underline u\leq T^n_\lambda \underline{u}\leq T^n_\lambda\bar u\leq \bar u.\]
Now we show that $T^n_\lambda\underline u$ is equi-continuous for all $n\geq 1$ and $\lambda<\lambda_0$. Let $\omega$ be a modulus of continuity of $\ell$ restricted to $X\times X\times [-M,M]^2$ where $M>\lambda_0 \max(\|\bar u\|_\infty,  \|\underline u \|_\infty)$. By symmetry of the roles of $x$ and $y$, assume without loss of generality $T^n_\lambda \underline u(x)\geq T^n_\lambda \underline u(y)$, let $z$ be  a minimal point in the definition of  $T_\lambda(T^{n-1}_\lambda \underline u)(y)$, then by (l1) we have
\begin{align*}
|T^n_\lambda \underline u(x)-T^n_\lambda \underline u(y)| &= T^n_\lambda \underline u(x)-T^n_\lambda \underline u(y) \\
&\leq \ell\big(z,x,\lambda T^{n-1}_\lambda \underline u(z),\lambda T^n_\lambda \underline u(x)\big)-\ell\big(z,y,\lambda T^{n-1}_\lambda \underline u(z),\lambda T^n_\lambda \underline u(y)\big)
\\ &\leq \ell\big(z,x,\lambda T^{n-1}_\lambda \underline u(z),\lambda T^n_\lambda \underline u(y)\big)-\ell\big(z,y,\lambda T^{n-1}_\lambda \underline u(z),\lambda T^n_\lambda \underline u(y)\big)
\\ &\leq \omega \big(d(x,y)\big).
\end{align*}

 We finally get the equi-continuity of $T^n_\lambda\underline u$. Then $T^n_\lambda \underline u$ uniformly converges to a function $\tilde u_\lambda$. We have
\begin{align*}
\|T_\lambda \tilde u_\lambda-\tilde u_\lambda\|_\infty &\leq \|T_\lambda \tilde u_\lambda-T^n_\lambda \underline{u}\|_\infty+\|T^n_\lambda \underline{u}-\tilde u_\lambda\|_\infty
\\ &\leq \|\tilde u_\lambda-T^{n-1}_\lambda \underline{u}\|_\infty+\|T^n_\lambda \underline{u}-\tilde u_\lambda\|_\infty\to 0.
\end{align*}
Then $\tilde u_\lambda$ is a fixed point of $T_\lambda$, and $\underline u\leq \tilde u_\lambda\leq \bar u$.

\medskip
\noindent {\bf Step 2.} We prove for $\lambda<\lambda_0$, all such fixed points $u_\lambda$ are uniformly bounded, more precisely, $\underline u\leq  u_\lambda\leq \bar u$.
. We prove that $u_\lambda\leq \bar u$, the lower bound of $u_\lambda$ is similar. Assume there is $x_0\in X$ such that
\[u_\lambda(x_0)-\bar u(x_0)=\max_{x\in X}\big(u_\lambda(x)-\bar u(x)\big)>0.\]
Let $(x_{-k})_{k\in \mathbb N}$ be a sequence obtained inductively such that for all $k\geqslant 0 $, $x_{-k-1}$ is  a point realizing the minimum in the definition of $T_0\bar u(x_{-k})$. It follows that for all $k\geqslant 0$, equatlity $\bar u(x_{-k}) -\bar u(x_{-k-1}) = \ell(x_{-k-1}, x_{-k},0,0)$ holds. We first show that if $\lambda<\lambda_0$, we have $u_\lambda(x_{-k})>\bar u(x_{-k})$ for all $k\geq 0$. Assume $u_\lambda(x_{-1})\leq \bar u(x_{-1})$. By (l1) and (l2), we have
\begin{align*}
u_\lambda(x_0)-u_\lambda(x_{-1})&\leq \ell\big(x_{-1},x_0,\lambda u_\lambda(x_{-1}),\lambda u_\lambda(x_0)\big)
\\ &\leq \ell\big(x_{-1},x_0,\lambda u_\lambda(x_{-1}),\lambda \bar u(x_0)\big)
\\ &\leq \ell\big(x_{-1},x_0,\lambda \bar u(x_{-1}),\lambda \bar u(x_0)\big)+\lambda \kappa_u(\bar u-u_\lambda)(x_{-1})
\\ &\leq \ell(x_{-1},x_0,0,0)+\lambda \kappa_u(\bar u-u_\lambda)(x_{-1})
\\ &=\bar u(x_0)-\bar u(x_{-1})+\lambda \kappa_u(\bar u-u_\lambda)(x_{-1}),
\end{align*}
which implies that
\[(1-\lambda\kappa_u)(u_\lambda-\bar u)(x_{-1})\geq u_\lambda(x_0)-\bar u(x_0)>0,\]
which leads to a contradiction as $0< 1-\lambda\kappa_u<1$. Then $u_\lambda(x_{-1})>\bar u(x_{-1})$.

Note that  by (l2), we have
\begin{equation*}\label{u=v01}
\begin{aligned}
\bar u(x_{0})-\bar u(x_{-1})&=\ell(x_{-1},x_{0},0,0)
\\ &\geq \ell\big(x_{-1},x_{0},\lambda \bar u(x_{-1}),\lambda \bar u(x_{0})\big)
\\ &\geq \ell\big(x_{-1},x_{0},\lambda u_\lambda(x_{-1}),\lambda u_\lambda(x_{0})\big)
\\ &\geq u_\lambda(x_{0})-u_\lambda(x_{-1}),
\end{aligned}
\end{equation*}
which implies
$u_\lambda(x_{-1})-\bar u(x_{-1})\geq u_\lambda(x_0)-\bar u(x_{0})$.
By the definition of $x_0$, this must be an equality, that is, $u_\lambda(x_{-1})-\bar u(x_{-1})=u_\lambda(x_0)-\bar u(x_{0})$. Moreover, as all previous inequalities are equalities, we obtain that $\ell(x_{-1},x_{0},0,0)
=\ell\big(x_{-1},x_{0},\lambda \bar u(x_{-1}),\lambda \bar u(x_{0})\big)$.

By induction, the same proof then shows that
\begin{equation*}\label{=max0}
u_\lambda(x_{-k})-\bar u(x_{-k})=\max_{x\in X}\big(u_\lambda(x)-\bar u(x)\big)>0,\quad \forall k\geq 0,
\end{equation*}
and that $\ell(x_{-k-1},x_{-k},0,0)
=\ell\big(x_{-k-1},x_{-k},\lambda \bar u(x_{-k-1}),\lambda \bar u(x_{-k})\big)$ for all $k\geqslant 0$.

Define the probability measure on $X\times X$, for $N>0$,
\begin{equation*}\label{defmu0}
\mu_N:=N^{-1}\sum_{k=-N}^{-1}\delta_{(x_k,x_{k+1})}.
\end{equation*}
By weak compactness of measures on $X\times X$ let $N_n \to +\infty$ be an extraction and $\mu$ a probability measure on $X\times X$ such that $\mu_{N_n}\to \mu$, as $n\to+\infty$.

Let $f\in C^0(X,\R)$. Since
\[\int_{X\times X} \big(f(x)-f(y)\big)d\mu_N=\frac{f(x_0)-f(x_{-N})}{N}\leq \frac{2\|f\|_\infty}{N}\to 0,\]
the measure $\mu$ is closed. We also have
\[\int_{X\times X}\ell(z,x,0,0)d\mu_N=\frac{\bar u(x_0)-\bar u(x_{-N})}{N}\to 0.\]
Thus, $\mu$ is a Mather measure. Since $\mu_{N_n}\to \mu$, for each $(z,x)\in \textrm{supp}(\mu)$, there is a sequence $(z_n,x_n)\in \textrm{supp}(\mu_{N_n})$ with $(z_n,x_n)\to (z,x)$. We have known that $u_\lambda(z_n)-\bar u(z_n)$ equals a constant $M>0$. Therefore, $u_\lambda(z)-\bar u(z)=M>0$. Similarly, we have $u_\lambda(x)-\bar u(x)=M>0$.

By the same argument, from $\ell(z_{n},x_{n},0,0)
=\ell\big(z_{n},x_{n},\lambda \bar u(z_{n}),\lambda \bar u(x_{n})\big)$ we obtain $\ell(z,x,0,0)
=\ell\big(z,x,\lambda \bar u(z),\lambda \bar u(x)\big)$.

Since $u_\lambda>\bar u\geq 0$ on supp$(\mu)$, by (l2), we have
\[\ell(z,x,r,s)=\ell\big(z,x,\lambda u_\lambda(z),\lambda u_\lambda (x)\big),\]
for all $(z,x)\in \textrm{supp}(\mu)$ and $r\in [0,\lambda u_\lambda(z)]$, $s\in[0,\lambda u_\lambda(x)]$. Then
\[\partial_{u,v}\ell(z,x,0,0)=0,\quad \forall (z,x)\in \textrm{supp}(\mu),\]
which contradicts (l4).

\medskip
\noindent {\bf Step 3.} We finally prove the equi-continuity of $u_\lambda$. Let $x,y\in X$. Assume without loss of generality that $u_\lambda(x)\geq u_\lambda(y)$, let $z$ the a minimal point of $u_\lambda(y)$, then by (l1) we have
\begin{align*}
u_\lambda(x)-u_\lambda(y)&\leq \ell\big(z,x,\lambda u_\lambda(z),\lambda u_\lambda(x)\big)-\ell\big(z,y,\lambda u_\lambda(z),\lambda u_\lambda(y)\big)
\\ &\leq \ell\big(z,x,\lambda u_\lambda(z),\lambda u_\lambda(y)\big)-\ell\big(z,y,\lambda u_\lambda(z),\lambda u_\lambda(y)\big)
\\ &\leq \omega \big(d(x,y)\big).
\end{align*}
\end{proof}

\section{Vanishing discount convergence}\label{secconv}

Let $\lambda\to 0$, by Proposition \ref{bdd}, there is a sequence $\lambda_n\to 0$ such that $u_{\lambda_n}$ uniformly converges. Let $u_*$ be a limit function of the family $(u_\lambda)_{\lambda\in(0,\lambda_0)}$. The vanishing discount problem concerns the uniqueness of $u_*$. 

\underline{\bf Notation}: Let us define $\ell_0 : X\times X\to \R$ the function defined by $(z,x)\mapsto \ell(z,x,0,0)$.

Let $\mathcal S_0$ be the set of subsolutions $w$ of $\ell_0$ that satisfy
\begin{equation}\label{cond}
\int_{X\times X}\big(\partial_u \ell(z,x,0,0)w(z)+\partial_v\ell(z,x,0,0)w(x)\big)d\mu(z,x)\geq 0,
\end{equation}
for all Mather measures $\mu$ of $\ell_0$. The set $\mathcal S_0$ is non-empty since negative weak KAM solutions fullfill (\ref{cond}).

\begin{result}\label{m1}
Let $\lambda\to 0$, $u_\lambda$ uniformly converges to
\[u_0:=\sup_{w\in\mathcal S_0}w,\]
where the supremum is taken pointwise. The function $u_0$ is therefore
 a fixed point of $T_0$.
\end{result}
 We also establish an alternative formula for the limit $u_0$:

 \begin{result}\label{alternative}
 The following holds for all $x\in X$:
 \[u_0(x)=\min_{\mu\in\mathfrak M_0}\frac{\int_{X\times X}\Big(\partial_u\ell(z,y,0,0)h(z,x)+\partial_v\ell(z,y,0,0)h(y,x)\Big)d\mu(z,y)}
{ \int_{X\times X}\Lambda(z,y)d\mu(z,y)},\]
where $\mathfrak M_0$ denotes the set of Mather measures of $\ell_0$,
\[\Lambda(z,y):=\partial_u \ell(z,y,0,0)+\partial_v\ell(z,y,0,0),\]
and $h(z,x)$ is  Peierls' barrier of $\ell_0$.

 \end{result}

\begin{remark}\label{barr} \rm
\begin{enumerate}
\item When $\ell$ satisfies that $\partial_v \ell (\cdot,\cdot,0,0)$ is constant, the previous equality reduces to
\[u_0(x)=\min_{\mu\in\mathfrak M_0}\bigg(\int_{X\times X}\Lambda(z,y)d\mu(z,y)\bigg)^{-1}\int_{X\times X}\Lambda(z,y)h(z,x)d\mu(z,y).\]

\item Symmetrically, when  $\ell$ satisfies that $\partial_u \ell (\cdot,\cdot,0,0)$ is constant, then\[u_0(x)=\min_{\mu\in\mathfrak M_0}\bigg(\int_{X\times X}\Lambda(y,z)d\mu(y,z)\bigg)^{-1}\int_{X\times X}\Lambda(z,y)h(y,x)d\mu(z,y).\]

\end{enumerate}
\end{remark}

\begin{proposition}\label{u*<}
For each Mather measure $\mu$ of $\ell_0$, we have
\[\int_X \big(\partial_u \ell(z,x,0,0)u_*(z)+\partial_v \ell(z,x,0,0)u_*(x)\big)d\mu(z,x)\geq 0.\]
\end{proposition}
\begin{proof}
Let $\mu$ be a Mather measure,
Since $T_\lambda u_\lambda=u_\lambda$, recalling that $\int \ell_0 d\mu = 0$, we have
\begin{align*}
&\int_{X\times X}\big(u_\lambda(x)-u_\lambda(z)\big)d\mu(z,x)
\\ &\leq \int_{X\times X}\ell\big(z,x,\lambda u_\lambda(z),\lambda u_\lambda(x)\big)d\mu(z,x)
\\ &\leq \int_{X\times X}\big(\ell(z,x,0,0)+\lambda\partial_u \ell(z,x,0,0)u_\lambda(z)+\lambda\partial_v \ell(z,x,0,0)u_\lambda(x)\big)d\mu(z,x)+\lambda \varepsilon(\lambda)
\\ &=\int_{X\times X}\big(\lambda\partial_u \ell(z,x,0,0)u_\lambda(z)+\lambda\partial_v \ell(z,x,0,0)u_\lambda(x)\big)d\mu(z,x)+\lambda \varepsilon(\lambda),
\end{align*}
where $\varepsilon(\lambda)=\|u_\lambda\|_\infty\eta(\lambda\|u_\lambda\|_\infty)$. Since $\mu$ is closed, we have
\[\int_{X\times X}\big(u_\lambda(x)-u_\lambda(z)\big)d\mu(z,x)=0.\]
Therefore, we have
\[\int_{X\times X} \big(\partial_u \ell(z,x,0,0)u_\lambda(z)+\partial_v \ell(z,x,0,0)u_\lambda(x)\big) d\mu(z,x)\geq -\varepsilon(\lambda).\]
Letting $\lambda\to 0$ along the sequence $(\lambda_n)_n$, we then get the result.
\end{proof}

\begin{remark}\label{xn}\rm
As observed in the proof of Proposition \ref{bdd}, since $u_\lambda$ is a fixed point of $T_\lambda$, for each $x\in X$, there is a sequence $(x_n)_{-n\in\mathbb N}$ with $x_0=x$, such that
\[\forall n\leqslant 0, \quad u_\lambda(x_n)=u_\lambda(x_{n-1})+\ell\big(x_{n-1},x_n,\lambda u_\lambda(x_{n-1}),\lambda u_\lambda(x_n)\big).\]
Here we note that the sequence $(x_n)_{-n\in \mathbb N}$ depends on $x$ and $\lambda$.
\end{remark}

\begin{lemma}\label{beta}
For $-n\in\mathbb N_+$, we define
\[\beta_n=\frac{\prod_{i=n}^{-1}\big(1+\lambda \partial_u \ell(x_{i},x_{i+1},0,0)\big)}{\prod_{i=n}^0\big(1-\lambda \partial_v \ell(x_{i-1},x_i,0,0)\big)},\quad \beta_0=\frac{1}{1-\lambda \partial_v \ell(x_{-1},x,0,0)}.\]
Since $(x_n)_{-n\in \mathbb N}$ depends on $x$ and $\lambda$, the sequence $(\beta_n)_{-n\in\mathbb N}$ also depends on $x$ and $\lambda$. For each integer $N>0$, we have
\begin{equation}\label{ul=}
\begin{aligned}
u_\lambda(x)&=\sum_{n=-N+1}^{0}\beta_n\bigg(\ell(x_{n-1},x_n,0,0)+\theta_k(\lambda)\bigg)
\\ &\quad +\big(1+\lambda\partial_u \ell(x_{-N},x_{-N+1},0,0)\big)\beta_{-N+1}u_\lambda(x_{-N}),
\end{aligned}
\end{equation}
where $|\theta_k(\lambda)|\leq \lambda\varepsilon(\lambda)$.
\end{lemma}
\begin{proof}
Since $u_\lambda$ is a fixed point of $T_\lambda$, we have
\begin{align*}
u_\lambda(x)&=u_\lambda(x_{-1})+\ell\big(x_{-1},x,\lambda u_\lambda(x_{-1}),\lambda u_\lambda(x)\big)
\\ &=u_\lambda(x_{-1})+\ell(x_{-1},x,0,0)+\partial_u \ell(x_{-1},x,0,0)\lambda u_\lambda(x_{-1})+\partial_v \ell(x_{-1},x,0,0)\lambda u_\lambda(x)+\theta_0(\lambda),
\end{align*}
where $|\theta_0(\lambda)|\leq \lambda \varepsilon(\lambda)$, which implies
\begin{align*}
u_\lambda(x)&=\frac{1+\lambda\partial_u \ell(x_{-1},x,0,0)}{1-\lambda \partial_v \ell(x_{-1},x,0,0)}u_\lambda(x_{-1})+\frac{1}{1-\lambda \partial_v \ell(x_{-1},x,0,0)}\big(\ell(x_{-1},x,0)+\theta_0(\lambda)\big)
\\ &=\big(1+\lambda\partial_u \ell(x_{-1},x,0,0)\big)\beta_0u_\lambda(x_{-1})+\beta_0\big(\ell(x_{-1},x,0)+\theta_0(\lambda)\big).
\end{align*}
We also have
\[u_\lambda(x_{-1})=u_\lambda(x_{-2})+\ell\big(x_{-2},x_{-1},\lambda u_\lambda(x_{-2}),\lambda u_\lambda(x_{-1})\big),\]
which implies
\begin{align*}
&\big(1+\lambda\partial_u \ell(x_{-1},x,0,0)\big)\beta_0u_\lambda(x_{-1})
\\ &=\big(1+\lambda\partial_u \ell(x_{-2},x_{-1},0,0)\big)\beta_{-1}u_\lambda(x_{-2})+\beta_{-1}\big(\ell(x_{-2},x_{-1},0,0)+\theta_{-1}(\lambda)\big),
\end{align*}
where $|\theta_{-1}(\lambda)|\leq \lambda\varepsilon(\lambda)$. Letting this procedure go on, and adding all equalities up, we get (\ref{ul=}).
\end{proof}

\begin{proposition}
There is $r>0$ such that for each $\lambda\in (0,r)$, there is $K>0$, independent of $\lambda$ and $x$, such that $\lambda\sum_{k\leq 0}\beta_k\leq K$.
\end{proposition}
\begin{proof}
We argue by contradiction. Assume there is a sequence $(\lambda_n)_{n\in \mathbb N}\to 0$ and $(x_n)_{n\in \mathbb N}\in X^{\mathbb N}$, and for all $n$, a minimal sequence $(x^n_k)_{-k\in\mathbb N}$ associated to $u_{\lambda_n}(x_n)$ and an integer $N_n>0$ such that \[\lambda_n\sum_{k=-N_n}^{-1}\beta^n_{k+1}\to +\infty.\]
Here $(\beta^n_k)_{-k\in\mathbb N}$ is the sequence associated to $(x^n_k)_{-k\in\mathbb N}$ as defined in Lemma \ref{beta}, which depends on $x_n$ and $\lambda_n$. As for $k$ fixed,  $\beta^n_k \to  1$ as $n\to +\infty$, we have $N_n\to +\infty$.

Define the probability measure
\[\mu_n:=C_n^{-1}\sum_{k=-N_n}^{-1}\beta_{k+1}^n\delta_{(x^n_{k},x^n_{k+1})},\]
where $C_n=\sum_{k=-N_n}^{-1}\beta^n_{k+1}$. Up to an extraction, we assume $\mu_n\to \mu$.

\noindent {\bf $\mu$ is closed:} Let $f\in C^0(X,\R)$, we have
\begin{align*}
&\bigg|\int_{X\times X}\big(f(x)-f(z)\big)d\mu_n(z,x)\bigg|
\\ &=C_n^{-1}\bigg|\sum_{k=-N_n}^{-1}\beta_{k+1}^n\big(f(x^n_{k+1})-f(x^n_k)\big)\bigg|
\\ &=C_n^{-1}\bigg|\sum_{k=-N_n}^{-1}(\beta_k^n-\beta_{k+1}^n)f(x^n_{k})-\beta^n_{-N_n}f(x^n_{-N_n})+\beta^n_0 f(x_n)\bigg|
\\ &\leq C_n^{-1}\bigg(\sum_{k=-N_n}^{-1}(\beta_{k+1}^n-\beta_{k}^n)\|f\|_\infty+2\|f\|_\infty\bigg)\leq 4C^{-1}_n\|f\|_\infty,
\end{align*}
where we use the fact $\beta^n_{k}\leq \beta^n_{k+1}$. Since $C_n\to +\infty$, $\mu$ is closed.

\noindent {\bf $\mu$ is minimizing:} By definition we have
\begin{align*}
&\bigg|\int_{X\times X}\ell(z,x,0,0)d\mu_n(z,x)\bigg|
\\ &=C_n^{-1}\bigg|\sum_{k=-N_n}^{-1}\beta_{k+1}^n\ell(x^n_k,x^n_{k+1},0,0)\bigg|
\\ &=C_n^{-1}\bigg|u_{\lambda_n}(x_n)-\sum_{k=-N_n}^{-1}\beta^n_{k+1}\theta^n_{k+1}(\lambda_n)-\big(1+\lambda\partial_u \ell(x^n_{-N},x^n_{-N+1},0,0)\big)\beta^n_{-N+1}u_\lambda(x^n_{-N})\bigg|
\\ &\leq 2C^{-1}_n\|u_{\lambda_n}\|_\infty+\lambda_n\omega(\lambda_n)\to 0.
\end{align*}

Now using
\[\frac{1+x}{1-y}=1+\frac{x+y}{1-y}\leq \exp\bigg\{\frac{x+y}{1-y}\bigg\}\leq \exp\{x+y\},\quad \textrm{for}\ \  y\leq 0,\]
we get
\begin{align*}
-\int_{X\times X}\Lambda(z,x)d\mu_n(z,x)&=-C_n^{-1}\sum_{k=-N_n}^{-1}\beta^n_{k+1}\Lambda(x^n_k,x^n_{k+1})
\\ &\leq -C_n^{-1}\sum_{k=-N_n}^{-1}\exp\{\lambda_n\sum_{i=k+1}^{-1}\Lambda(x^n_{i},x^n_{i+1})\}\Lambda(x^n_k,x^n_{k+1})
\\ &\leq -C_n^{-1}\exp\{\|\Lambda\|_\infty\}\sum_{k=-N_n}^{-1}\exp\{\lambda_n\sum_{i=k}^{-1}\Lambda(x^n_{i},x^n_{i+1})\}\Lambda(x^n_k,x^n_{k+1})
\\ &\leq C_n^{-1}\exp\{\|\Lambda\|_\infty\}\int_0^{+\infty}e^{-\lambda_n x}dx=\frac{\exp\{\|\Lambda\|_\infty\}}{\lambda_n C_n}\to 0,
\end{align*}
which contradicts (l4).
\end{proof}

In the following, let $(x_n)_{-n\in \mathbb N}$ and $(\beta_n)_{-n\in\mathbb N}$ be the sequences defined in Remark \ref{xn} and Lemma \ref{beta} respectively, associated to the pair $x_0\in X$ and $\lambda \in (0,r)$.
\begin{proposition}
For each $\lambda\in (0,r)$ and $x_0\in X$, we have
\begin{equation}\label{u==}
  u_{\lambda}(x_0)=\sum_{n\leq 0}\beta_n \ell(x_{n-1},x_n,0,0)+\Omega(\lambda),
\end{equation}
where
\[\lim_{\lambda\to 0}\Omega(\lambda)\to 0.\]
\end{proposition}
\begin{proof}
Since $\lambda\sum_{n\leq 0}\beta_n\leq K$, we have $\beta_{-N+1}\to 0$ and
\[\bigg|\sum_{n=-N+1}^0\beta_n\theta_n(\lambda)\bigg|\leq \sum_{n\leq 0}\beta_n\lambda \varepsilon(\lambda)\leq K\varepsilon(\lambda)\to 0.\]
By (\ref{ul=}) we get (\ref{u==}), where $\Omega(\lambda):=\sum_{n\leq 0}\beta_n\theta_n(\lambda)$.
\end{proof}

\begin{definition} \rm
We define the following probability measures on $X\times X$
\[\mu^1_\lambda=C_\lambda^{-1} \sum_{k\leq -1}\frac{\beta_{k+1}}{1-\lambda\partial_ v\ell(x_{k-1},x_k,0,0)}\delta_{(x_k,x_{k+1})},\]
and
\[\mu^2_\lambda=C_\lambda^{-1} \sum_{k\leq -1}\frac{\beta_{k+1}}{1-\lambda\partial_ v\ell(x_{k-1},x_k,0,0)}\delta_{(x_{k-1},x_{k})},\]
where
\[C_\lambda:=\sum_{k\leq -1}\frac{\beta_{k+1}}{1-\lambda\partial_ v\ell(x_{k-1},x_k,0,0)}.\]
\end{definition}

Note that as $x_0$ will be fixed in what follows, we only specify explicitly the dependance of those measures in $\lambda$ but they also depend on $x_0$.


Since the functions $(z,x) \mapsto \partial_v \ell(z,x,0,0)$ and $(z,x) \mapsto \partial_u \ell(z,x,0,0)$ are bounded respectively by $\kappa_v$ and $\kappa_u$, it is easily observed that $C_\lambda \to +\infty$ (each term of the sum converges to $1$ as $\lambda \to 0$). Moreover  $ \lambda C_\lambda  \leq \lambda \sum_{k\leq 0}\beta_k\leq K$.

\begin{proposition}\label{u*>}
For each subsolution $w$ of $\ell_0$ and $\lambda\in (0,r)$, we have
\begin{align*}
u_{\lambda}(x)&\geq \beta_0w(x)
\\ &\quad +\lambda C_\lambda\int_{X\times X} \partial_u \ell(z,x,0,0)w(z)d\mu^1_\lambda(z,x)
\\ &\quad +\lambda C_\lambda\int_{X\times X}\partial_v \ell(z,x,0,0)w(x)d \mu^2_{\lambda}(z,x)+\Omega(\lambda).
\end{align*}
\end{proposition}
\begin{proof}
By (\ref{u==}) we have
\begin{equation*}\label{u>w}
\begin{aligned}
u_\lambda(x)&\geq \sum_{n\leq 0}\beta_n \big(w(x_{n})-w(x_{n-1})\big)+\Omega(\lambda)
\\ &=\beta_0w(x)+\lambda\sum_{k\leq -1}\frac{\beta_{k+1}}{1-\lambda\partial_v\ell(x_{k-1},x_k,0,0)}\partial_u \ell(x_{k},x_{k+1},0,0)w(x_k)
\\ &\quad +\lambda\sum_{k\leq -1}\frac{\beta_{k+1}}{1-\lambda\partial_v\ell(x_{k-1},x_k,0,0)}\partial_v \ell(x_{k-1},x_{k},0,0)w(x_k)+\Omega(\lambda)
\\ &\geq \beta_0w(x)+\lambda C_\lambda\int_{X\times X} \partial_u \ell(z,x,0,0)w(z)d\mu^1_\lambda
\\ &\quad +\lambda C_\lambda\int_{X\times X}\partial_v \ell(z,x,0,0)w(x)d \mu^2_{\lambda}(z,x)+\Omega(\lambda).
\end{aligned}
\end{equation*}
\end{proof}

\begin{lemma}\label{samelimit}
The limits of $\mu^1_\lambda$ and $\mu^2_\lambda$ coincide in the weak* topology as $\lambda\to 0$. That is, if there is a sequence $\lambda_n\to 0$ such that $\mu^1_{\lambda_n}\to \mu$, then $\mu^2_{\lambda_n}\to \mu$.
\end{lemma}
\begin{proof}
For all $f\in C^0(X\times X,\R)$, we have
\begin{equation*}
\begin{aligned}
&\bigg|\int_{X\times X}f(z,x)d(\mu^1_\lambda-\mu^2_\lambda)\bigg|
\\ &=C_\lambda^{-1}\bigg|\sum_{k\leq -1}\frac{\beta_{k+1}}{1-\lambda\partial_ v\ell(x_{k-1},x_k,0,0)}\big(f(x_k,x_{k+1})-f(x_{k-1},x_k)\big)\bigg|
\\ &=C_\lambda^{-1}\bigg|\frac{\beta_{0}}{1-\lambda\partial_ v\ell(x_{-2},x_{-1},0,0)}f(x_{-1},x_0)+\sum_{k\leq -2}\frac{\beta_{k+1}}{1-\lambda\partial_ v\ell(x_{k-1},x_k,0,0)}f(x_k,x_{k+1})
\\ &\quad -\sum_{k\leq -2}\frac{\beta_{k+1}}{1+\lambda\partial_u \ell(x_{k+1},x_{k+2},0,0)}f(x_k,x_{k+1})\bigg|
\\ &\leq C_\lambda^{-1} \|f\|_\infty
\\ &\quad +\lambda C_\lambda^{-1}\bigg|\sum_{k\leq -2}\frac{\partial_u \ell(x_{k+1},x_{k+2},0,0)+\partial_v\ell(x_{k-1},x_k,0,0)}{\big(1-\lambda\partial_ v\ell(x_{k-1},x_k,0,0)\big)\big(1+\lambda\partial_u \ell(x_{k+1},x_{k+2},0,0)\big)}\beta_{k+1}f(x_k,x_{k+1})\bigg|
\\ &\leq C_\lambda^{-1} \|f\|_\infty+\lambda C_\lambda^{-1} \frac{\kappa_u+\kappa_v}{1-\lambda\kappa_u}\sum_{k\leq -1}\beta_{k+1}\|f\|_\infty
\\ &\leq \left(1+\frac{\kappa_u+\kappa_v}{1-\lambda\kappa_u}K\right)C_\lambda^{-1}\|f\|_\infty\to 0.
\end{aligned}
\end{equation*}
Indeed, recall that $C_\lambda\to +\infty$ as $\lambda \to 0$.
\end{proof}

\begin{proposition}\label{mathermeasure}
Any limit $\mu$ of $d\mu^{1,2}_{\lambda_n}$ as $\lambda_n\to 0$ is a Mather measure of $\ell(z,x,0,0)$.
\end{proposition}
\begin{proof}
We first prove that $\mu$ is {\bf closed}. Let $f\in C^0(X,\R)$, then
\begin{align*}
&\int_{X\times X}\big(f(x)-f(z)\big)d\mu^1_\lambda(z,x)
\\ &=C_\lambda^{-1} \sum_{k\leq -1}\frac{\beta_{k+1}}{1-\lambda\partial_ v\ell(x_{k-1},x_k,0,0)}\big(f(x_{k+1})-f(x_k)\big)
\\ &=C_\lambda^{-1} \frac{\beta_{0}}{1-\lambda\partial_ v\ell(x_{-2},x_{-1},0,0)}f(x)
\\ &\quad +\lambda C_\lambda^{-1} \sum_{k\leq -1}\beta_{k+1}\frac{\partial_ u\ell(x_{k-2},x_{k-1},0,0)+\partial_ v\ell(x_{k},x_{k+1},0,0)}{\big(1-\lambda\partial_ v\ell(x_{k-2},x_{k-1},0,0)\big)\big(1-\lambda\partial_ v\ell(x_{k-1},x_{k},0,0)\big)}f(x_k).
\end{align*}
Since $-\kappa_v<\partial_v \ell(z,x,0,0)\leq 0$ and $-\kappa_v\leq \partial_u \ell(z,x,0,0)\leq 0$, we have
\begin{multline*}
\bigg|\int_{X\times X}\big(f(x)-f(z)\big)d\mu^1_\lambda(z,x)\bigg|\\
\leq C_\lambda^{-1}\|f\|_\infty(1+(\kappa_u+\kappa_v)\lambda\sum_{k\leq 0}\beta_k)\\
\leq (1+(\kappa_u+\kappa_v)K)C_\lambda^{-1}\|f\|_\infty\to 0.
\end{multline*}
Then we prove $\mu$ is {\bf minimizing}. By (\ref{u==}) we have
\begin{align*}
\bigg|\int_{X\times X}  &\ell(z,x,0,0)\big(1-\lambda\partial_v \ell_0(z,x)\big)d\mu^2_\lambda(z,x)-C_\lambda^{-1}u_\lambda(x_0)\bigg|
\\ &=C_\lambda^{-1}\bigg|\sum_{k\leq -1}(\beta_{k+1}-\beta_k)\ell(x_{k-1},x_{k},0,0)-\beta_0 \ell_0(x_{-1},x)-\Omega(\lambda)\bigg|
\\ &\leq C_\lambda^{-1}\bigg(\bigg|\lambda\sum_{k\leq -1}\frac{\partial_u \ell_0(x_{k},x_{k+1},0,0)+\partial_v \ell(x_{k-1},x_k,0,0)}{1-\lambda\partial_v \ell(x_{k-1},x_k,0,0)}\beta_{k+1}\ell_0(x_{k-1},x_{k})\bigg| \\
&\qquad \qquad+\|\ell_0\|_\infty+K\varepsilon(\lambda)\bigg)
\\ &\leq C_\lambda^{-1}\big(\lambda \sum_{k\leq 0}(\kappa_u+\kappa_v)\beta_k\|\ell_0\|_\infty+\|\ell_0\|_\infty+K\varepsilon(\lambda)\big)
\\
&\leq C_\lambda^{-1}\big(((\kappa_u+\kappa_v)K+1)\|\ell_0\|_\infty+K\varepsilon(\lambda)\big)\to 0.
\end{align*}
Let $\lambda\to 0$, we get $\int_{X\times X}\ell_0(z,x)d\mu_\lambda(z,x)=0$.
\end{proof}

\bigskip

We now turn to the proof of the main theorem. Recall that  $\mathcal S_0$ is the set of subsolutions $w$ of $\ell_0$ that satisfy
\begin{equation*}
\int_{X\times X}\big(\partial_u \ell(z,x,0,0)w(z)+\partial_v\ell(z,x,0,0)w(x)\big)d\mu(z,x)\geq 0,
\end{equation*}
for all Mather measures $\mu$ of $\ell_0$.

\noindent \emph{ Proof of Theorem \ref{m1}.}
We first show that $u_0$ is well-define, that is,  functions in $\mathcal S_0$ are uniformly bounded from above. Assume there is $w\in\mathcal S_0$ such that $w\geq \delta>0$, then by (\ref{cond}) we have
\begin{align*}
0&\leq \int_{X\times X}\big(\partial_u \ell(z,x,0,0)w(z)+\partial_v\ell(z,x,0,0)w(x)\big)d\mu(z,x)
\\ &\leq \delta\int_{X\times X}\big(\partial_u \ell(z,x,0,0)+\partial_v\ell(z,x,0,0)\big)d\mu(z,x),
\end{align*}
which contradicts (l4). Therefore, for all $w\in\mathcal S_0$, there is $x_0\in X$ such that $w(x_0)\leq 0$. By the equi-continuity of subsolutions of $\ell_0$, the result follows.

Recall that we consider a decreasing sequence $\lambda_n \to 0$ such that $u_{\lambda_n} \to u_*$ uniformly.

Since for 	all $(x,z)\in X\times X$,
\[u_\lambda(x)-u_\lambda(z)\leq \ell\big(z,x,\lambda u_\lambda(x),\lambda u_\lambda(z)\big).\]
Let $\lambda\to 0$ we get
\[u_*(z)-u_*(x)\leq \ell(z,x,0,0),\]
which means $u_*$ is a subsolution of $\ell_0$. By Proposition \ref{u*<}, $u_*\in\mathcal S_0$, which implies $u_*\leq u_0$.

Let now $x_0\in X$ and up to a further extraction, assume that the associated measures converge. By Lemma \ref{samelimit} the limits are the same and  $ \lim\limits_{n\to +\infty} \mu_{\lambda_n}^1 =   \lim\limits_{n\to +\infty} \mu_{\lambda_n}^2 = \mu$ is a Mather measure by Proposition \ref{mathermeasure}.
For each $w\in\mathcal S_0$, by Proposition \ref{u*>}, we have
\begin{align*}
u_*(x_0)&\geq \lim_{n\to +\infty}\beta_0 w(x_0)
\\ &\quad +\limsup_{n\to +\infty}\lambda_n C_{\lambda_n}\bigg(\int_{X\times X} \partial_u \ell(z,x,0,0)w(z)d\mu^1_{\lambda_n}+\int_{X\times X}\partial_v \ell(z,x,0,0)w(x)d \mu^2_{\lambda_n}\bigg)\\
&=w(x_0) + \limsup_{n\to +\infty}\lambda_n C_{\lambda_n}\int_{X\times X} \big(\partial_u \ell(z,x,0,0)+\partial_v \ell(z,x,0,0)\big)w(x)d \mu
\\ &\geq w(x_0),
\end{align*}
where we have used that $w\in \mathcal S_0$.
Therefore, $u_*(x_0)\geq \sup_{w\in\mathcal S_0}w(x_0)= u_0(x_0)$. We finally get
$u_*=u_0$.

Now we prove that $u_0$ is a fixed point of $T_0$. We have seen that $u_0=u_*$ is a subsolution. Let $x_0 \in X$. Since $X$ is compact, let $z_n$ be a point realizing the minimum in $u_{\lambda_n}(x_0)$ and up to extracting, assume $z_{\lambda_n}\to z_*$. By
\[u_{\lambda_n}(x_0)-u_{\lambda_n}(z_{\lambda_n})=\ell\big(z_{\lambda_n},x_0,\lambda_n u_{\lambda_n}(z_{\lambda_n}),\lambda_n u_{\lambda_n}(x_0)\big),\]
we get
\[u_0(x_0)-u_0(z_*)=\ell(z_*,x_0,0,0).\]
Thus, $u_0$ is a fixed point of $T_0$.
\qed
 \begin{remark}\rm
 As a byproduct of the previous proof, we have also proven that $u_0\in \mathcal S_0$.

 \end{remark}

\bigskip
We finish this section by the alternative representation formula of $u_0$.

\noindent {\it Proof of Theorem \ref{alternative}.}
Define
 \[\hat u_0(x)=\min_{\mu\in\mathfrak M_0}\frac{\int_{X\times X}\Big(\partial_u\ell(z,y,0,0)h(z,x)+\partial_v\ell(z,y,0,0)h(y,x)\Big)d\mu(z,y)}
{ \int_{X\times X}\Lambda(z,y)d\mu(z,y)},\]

where $\mathfrak M_0$ denotes the set of Mather measures of $\ell_0$. Note first that for each $\mu\in \mathfrak M_0$,
 the function
$$x\mapsto \frac{\int_{X\times X}\Big(\partial_u\ell(z,y,0,0)h(z,x)+\partial_v\ell(z,y,0,0)h(y,x)\Big)d\mu(z,y)}
{ \int_{X\times X}\Lambda(z,y)d\mu(z,y)},$$
is a subsolution of $\ell_0$. Indeed each $-h(z,\cdot)$ is a subsolution (Proposition \ref{proph}) hence the integral is a barycenter of subsolutions (Proposition \ref{conv}). Hence $\hat u_0$ is also a subsolution for $\ell_0$ as an infimum of subsolutions.

\noindent \underline{{\bf proof that $u_0\leq \hat u_0$}}: Let $x\in X$ and $\mu\in \frak M_0$. Integrating the inequalities $ u_0(x)\leqslant  u_0 (z) + h(z,x)$ recalled in Proposition \ref{proph}, we find that
\begin{equation}\label{inequ}u_0(x)\int_{X\times X}\partial_u\ell(z,y,0,0)d\mu(z,y)\geq \int_{X\times X}\partial_u\ell(z,y,0,0)\big(u_0(z)+h(z,x)\big)d\mu(z,y),
\end{equation}
\begin{equation}\label{ineqv}
u_0(x)\int_{X\times X}\partial_v\ell(z,y,0,0)d\mu(z,y)\geq \int_{X\times X}\partial_v\ell(z,y,0,0)\big(u_0(y)+h(y,x)\big)d\mu(z,y).
\end{equation}
Recall that $u_0\in \mathcal S_0$ so that
$$\int_{X\times X}\big(\partial_u \ell(z,y,0,0)u_0(z)+\partial_v\ell(z,y,0,0)u_0(y)\big)d\mu(z,y)\geq 0.$$

Therefore, summing \eqref{inequ} and \eqref{ineqv} and using the previous inequality, we obtain
$$u_0(x)\int_{X\times X}\Lambda(z,y)d\mu(z,y)\geq \int_{X\times X}\Big(\partial_u\ell(z,y,0,0)h(z,x)+\partial_v\ell(z,y,0,0)h(y,x)\Big)d\mu(z,y).$$
Dividing by $\int_{X\times X}\Lambda(z,y)d\mu(z,y)<0$ and taking a minimum over $\mu\in \mathfrak M_0$ yields the desired $u_0(x)\leqslant \hat u_0(x)$.

%

\underline{{\bf proof  that $u_0\geq \hat u_0$}}: We first show $v_y(\cdot):=-h(\cdot,y)+\hat u_0(y)\in\mathcal S_0$ for all $y\in X$. Let $\mu \in \mathfrak M_0$.
We get
\begin{align*}
&\int_{X\times X}\big(\partial_u \ell(z,x,0,0)v_y(z)+\partial_v \ell(z,x,0,0)v_y(x)\big)d\mu(z,x)
\\ &\quad =-\int_{X\times X}\Big(h(z,y)\partial_u \ell(z,x,0,0) + h(x,y)\partial_v \ell(z,x,0,0)\Big) d\mu(z,x)
\\ &\quad\quad +\int_{X\times X}\Lambda(z,x)\hat u_0(y)d\mu(z,x)
\\&\quad =-\int_{X\times X}\Big(h(z,y)\partial_u \ell(z,x,0,0) + h(x,y)\partial_v \ell(z,x,0,0)\Big) d\mu(z,x)
\\&+\int_{X\times X}\Lambda(z,x)d\mu(z,x)
\min_{\tilde\mu\in\mathfrak M_0}\frac{\int_{X\times X}\Big(\partial_u\ell(z,x,0,0)h(z,y)+\partial_v\ell(z,x,0,0)h(x,y)\Big)d\tilde\mu(z,x)}
{ \int_{X\times X}\Lambda(z,x)d\tilde\mu(z,x)}
\\ & &{\geq 0}.
\end{align*}
It follows that $v_y\leqslant u_0$ and evaluating at $y$ yields
 $u_0(y)\geq -h(y,y)+\hat u_0(y)$. Let $y\in\mathfrak A$, we have $h(y,y)=0$, and $u_0(y)\geq \hat u_0(y)$. By comparison (Proposition \ref{comparison}), as $u_0$ is a solution and $\hat u_0$ a subsolution, we finally get $u_0\geq \hat u_0$.
\qed

\section{Uniqueness of $u_\lambda$}

\begin{result}\label{thm2}
The fixed point $u_\lambda$ of $T_\lambda$ is unique if $\lambda$ is small and one of the following holds
\begin{itemize}
\item[(1)] $\ell(z,x,u,v)$ is concave in $u$ and concave in  $v$;
\item[(2)] $\partial_{u,v}\ell(z,x,u,v)$ exist and are continuous for $(u,v)$ near $(0,0)$.
\end{itemize}
\end{result}
\begin{proof}
We argue by contradiction. Let $u_\lambda$ and $v_\lambda$ be two fixed points. Assume
\[u_\lambda(x_0)-v_\lambda(x_0)=\max_{x\in X}\big(u_\lambda(x)-v_\lambda(x)\big)>0.\]
Let $(x_{-k})_{k\in \mathbb N}$ be a minimizing sequence associated to  $v_\lambda(x_0)$ as defined in Remark \ref{xn}.

\noindent {\bf Step 1.} We first show that if $\lambda$ is small enough, we have $u_\lambda(x_{-k})>v_\lambda(x_{-k})$ for all $k\geq 0$. Assume $u_\lambda(x_{-1})\leq v_\lambda(x_{-1})$. By (l1) and (l2), we have
\begin{align*}
u_\lambda(x_0)-u_\lambda(x_{-1})&\leq \ell\big(x_{-1},x,\lambda u_\lambda(x_{-1}),\lambda u_\lambda(x_0)\big)
\\ &\leq \ell\big(x_{-1},x_0,\lambda u_\lambda(x_{-1}),\lambda v_\lambda(x_0)\big)
\\ &\leq \ell\big( x_{-1},x_0,\lambda v_\lambda(x_{-1}),\lambda v_\lambda(x_0)\big)+\lambda \kappa_u(v_\lambda-u_\lambda)(x_{-1})
\\ &=v_\lambda(x_0)-v_\lambda(x_{-1})+\lambda \kappa_u(v_\lambda-u_\lambda)(x_{-1}),
\end{align*}
which implies that
\[(1-\lambda\kappa_u)(u_\lambda-v_\lambda)(x_{-1})\geq u_\lambda(x_0)-v_\lambda(x_0)>0,\]
which leads to a contradiction. Then $u_\lambda(x_{-1})>v_\lambda(x_{-1})$. We then go on to find, using (l2), that

\begin{align*}
u_\lambda(x_{-1})-u_\lambda(x_{-2})&\leq \ell\big(x_{-2},x,\lambda u_\lambda(x_{-2}),\lambda u_\lambda(x_{-1})\big)
\\ &\leq \ell\big(x_{-2},x_{-1},\lambda u_\lambda(x_{-2}),\lambda v_\lambda(x_{-1})\big)
\\ &\leq \ell\big( x_{-2},x_{-1},\lambda v_\lambda(x_{-2}),\lambda v_\lambda(x_{-1})\big)
\\ &=v_\lambda(x_{-1})-v_\lambda(x_{-2}),
\end{align*}
therefore $u_\lambda(x_{-2})-v_\lambda(x_{-2}) \geq u_\lambda(x_{-1})-v_\lambda(x_{-1})>0$.

 By induction, we have $u_\lambda(x_{-k})>v_\lambda(x_{-k})$ for all $k\geq 0$.

\medskip
\noindent {\bf Step 2.} By Step 1 and (l2), we have
\begin{equation}\label{u=v}
\begin{aligned}
v_\lambda(x_{-k+1})-v_\lambda(x_{-k})&=\ell\big(x_{-k},x_{-k+1},\lambda v_\lambda(x_{-k}),\lambda v_\lambda(x_{-k+1})\big)
\\ &\geq \ell\big(x_{-k},x_{-k+1},\lambda u_\lambda(x_{-k}),\lambda u_\lambda(x_{-k+1})\big)
\\ &\geq u_\lambda(x_{-k+1})-u_\lambda(x_{-k}),
\end{aligned}
\end{equation}
which implies
\[u_\lambda(x_{-k})-v_\lambda(x_{-k})\geq \dots\geq u_\lambda(x_0)-v_\lambda(x_{0}).\]
By the definition of $x_0$, all inequalities above are equalities, that is,
\begin{equation}\label{=max}
u_\lambda(x_{-k})-v_\lambda(x_{-k})=\max_{x\in X}\big(u_\lambda(x)-v_\lambda(x)\big)>0,\quad \forall k\geq 0.
\end{equation}

\medskip

\noindent {\bf Step 3.} Define the probability measure
\begin{equation}\label{defmu}
\mu_N:=N^{-1}\sum_{k=-N}^{-1}\delta_{(x_k,x_{k+1})}.
\end{equation}
By compactness of measures, let $(N_n)_n$ be an extraction such that $\mu_{N_n} \to \mu$ as $ N\to+\infty$.
Since
\[\left| \int_{X\times X} \big(f(x)-f(y)\big)d\mu_N\right|=\left| \frac{f(x_0)-f(x_{-N})}{N}  \right| \leq \frac{2\|f\|_\infty}{N}\to 0,\]
the measure $\mu$ is closed. We also have
\[\int_{X\times X}\ell\big(z,x,\lambda v_\lambda(z),\lambda v_\lambda(x)\big)d\mu_N=\frac{v_\lambda(x_0)-v_\lambda(x_{-N})}{N}\to 0.\]
As all inequalities in (\ref{u=v}) are equalities, it follows that
\begin{equation}\label{u==v}
\int_{X\times X}\ell\big(z,x,\lambda v_\lambda(z),\lambda v_\lambda(x)\big)d\mu=\int_{X\times X}\ell\big(z,x,\lambda u_\lambda(z),\lambda u_\lambda(x)\big)d\mu=0.
\end{equation}
Since $\mu_{N_n}\to \mu$, for each $(z,x)\in \textrm{supp}(\mu)$, there is a sequence $(z_{n},x_{n})\in \textrm{supp}(\mu_{N_n})$ with $(z_n,x_n)\to (z,x)$. By Step 2, $u_\lambda(z_{n})-v_\lambda(z_n)$ equals a positive constant $M$. Therefore, $u_\lambda(z)-v_\lambda(z)=M>0$. Similarly, we have $u_\lambda(x)-v_\lambda(x)=M>0$. By (l2), we have
\[\ell(z,x,r,s)=\ell\big(z,x,\lambda u_\lambda(z),\lambda u_\lambda (x)\big),\]
for all $(z,x)\in \textrm{supp}(\mu)$ and $r\in [\lambda v_\lambda(z),\lambda u_\lambda(z)]$, $s\in[\lambda v_\lambda(x),\lambda u_\lambda(x)]$.

\medskip

\noindent {\bf Conclusion under hypothesis (1).} By (\ref{=max}), the intervals $[\lambda v_\lambda(z),\lambda u_\lambda(z)]$ and $[\lambda v_\lambda(x),\lambda u_\lambda(x)]$ have no-empty interior, by the concavity we have
\[\ell\big(z,x,\lambda u_\lambda(z),\lambda u_\lambda(x)\big)=\max_{r,s}\ell(z,x,r,s),\quad \forall (z,x)\in \textrm{supp}(\mu).\]
By (l2), we have
\[\ell(z,x,r,s)=\ell\big(z,x,\lambda u_\lambda(z),\lambda u_\lambda(x)\big),\quad \forall r\leq \lambda u_\lambda(z),\ s\leq \lambda u_\lambda(x).\]
Let $u_0<0$ and $u_0\leq \min_{x\in X}\lambda u_\lambda(x)$, we have
\[\int_{X\times X}\ell(z,x,u_0,u_0)d\mu=\int_{X\times X}\ell\big(z,x,\lambda u_\lambda(z),\lambda u_\lambda(x)\big)d\mu=0.\]
By (l2) we also have
\[0=\int_{X\times X}\ell(z,x,u_0,u_0)d\mu\geq \int_{X\times X}\ell(z,x,0,0)d\mu.\]
Thus, $\mu$ is a Mather measure of $\ell(z,x,0,0)$. By (l2) again we have
\[\ell(z,x,r,s)=\ell(z,x,u_0,u_0),\quad \forall (z,x)\in \textrm{supp}(\mu),\ \forall r,s\in [u_0,0].\]
Since $u_0<0$, we have
\[\partial_{u,v}\ell(z,x,0,0)=0,\quad \forall (z,x)\in \textrm{supp}(\mu),\]
which contradicts (l4).

\medskip

\noindent {\bf Conclusion under hypothesis (2).} Since the set of Mather measures is compact, there is $\epsilon>0$ such that
\[\int_{X\times X}\Lambda(z,x)d\mu(z,x)<-2\epsilon,\quad \forall \mu \in \mathfrak M_0.\]
We first show that there is $r>0$ and $N_0>0$ such that
\[N^{-1}\sum_{k=-N}^{-1}\Lambda(x_{k},x_{k+1})<-2\epsilon,\quad \forall \lambda\in (0,r),\ \forall N\geq N_0.\]
If not, we assume that there is a sequence $\lambda_n\to 0$ and $N_n\to +\infty$ such that
\[N_n^{-1}\sum_{k=-N_n}^{-1}\Lambda(x_{k},x_{k+1})\geq -2\epsilon.\]
Extracting a subsequence if necessary, let $\mu_{\lambda_n}$ be the limit given by (\ref{defmu}) and $\mu_{\lambda_n}\to \mu$. By (\ref{u==v}), $\mu$ is a Mather measure. We then get a contradiction.

Since $\partial_{u,v}\ell$ is continuous for $(u,v)$ near $(0,0)$, for $\lambda$ small and $N$ large, we have
\[N^{-1}\sum_{k=-N}^{-1} \Big[ \partial_u \ell\big(x_{k},x_{k+1},\lambda u_\lambda(x_k),\lambda u_\lambda(x_{k+1})\big)+\partial_v \ell\big(x_{k},x_{k+1},\lambda u_\lambda(x_k),\lambda u_\lambda(x_{k+1})\big)\Big]<-\epsilon.\]
Since all inequalities in (\ref{u=v}) are equalities, we have
\[\partial_u \ell\big(x_{k},x_{k+1},\lambda u_\lambda(x_k),\lambda u_\lambda(x_{k+1})\big)=\partial_v \ell\big(x_{k},x_{k+1},\lambda u_\lambda(x_k),\lambda u_\lambda(x_{k+1})\big)=0,\]
which leads to a contradiction.
\end{proof}

\section*{Acknowledgements}

The authors are supported by ANR CoSyDy (ANR-CE40-0014).








\bibliography{bib}
\bibliographystyle{siam}

\end{document}